\documentclass[11pt]{article}
\usepackage{hyperref}
\usepackage{times}  
\usepackage{mathpazo}
\usepackage{amssymb,amsmath,amsthm}
\usepackage{epsfig}

\newtheorem{theorem}{Theorem}[section]

\newtheorem{definition}[theorem]{Definition}

\newtheorem{claim}[theorem]{Claim}
\newtheorem{lemma}[theorem]{Lemma}

\newtheorem{remark}[theorem]{Remark}

\newcommand{\qedsymb}{\hfill{\rule{2mm}{2mm}}}
\renewenvironment{proof}[1][]{\begin{trivlist}
\item[\hspace{\labelsep}{\bf\noindent Proof#1:\/}] }{\qedsymb\end{trivlist}}

\def\Z{{\mathbb{Z}}}

\def\mod{\mbox{mod}}

\newcommand{\eps}{\epsilon}
\renewcommand{\epsilon}{\varepsilon}

\begin{document}

\title{{\bf Symmetric Complete Sum-free Sets in Cyclic Groups}}

\author{
Ishay Haviv\thanks{School of Computer Science, The Academic College of Tel Aviv-Yaffo, Tel Aviv 61083, Israel.}
\and Dan Levy\thanks{School of Computer Science, The Academic College of Tel Aviv-Yaffo, Tel Aviv 61083, Israel.}
}


\maketitle

\begin{abstract}
We present constructions of symmetric complete sum-free sets in general finite cyclic groups.
It is shown that the relative sizes of the sets are dense in $[0,\frac{1}{3}]$, answering a question of Cameron,
and that the number of those contained in the cyclic group of order $n$ is exponential in $n$.
For primes $p$, we provide a full characterization of the symmetric complete sum-free subsets of $\Z_p$ of size at least $(\frac{1}{3}-c) \cdot p$, where $c>0$ is a universal constant.
\end{abstract}

\section{Introduction}

This paper studies symmetric complete sum-free subsets of finite cyclic groups.
Let $G$ be an abelian additive group.
A set $S \subseteq G$ is \emph{symmetric} if $x\in S$ implies $-x\in S$.
It is \emph{complete} if for every $z\in G\setminus S$ there exist $x,y\in S$ for which $x+y=z$,
and it is \emph{sum-free} if there are no $x,y,z\in S$ satisfying $x+y=z$.

Sum-free sets were first studied by Schur~\cite{Schur1916}, who proved in 1916 that the set of positive integers cannot be partitioned into a finite number of sum-free sets.
This result initiated a fruitful line of research, which now, over a century later, is central in the area of additive combinatorics, with applications to several areas of mathematics such as extremal graph theory, Ramsey theory, projective geometry, and model theory (see, e.g.,~\cite{Cameron84,Cameron87,Payne95,CalkinC98}).

A fundamental question on sum-free sets is how large can a sum-free subset of a group $G$ be.
It is easy to see that the size of a largest sum-free subset of $\Z$ contained in $[1,n]$ is $\lceil n/2 \rceil$, attained, for example, by the set of all odd numbers in this range.
For the cyclic group $G = \Z_p$ where $p$ is a prime, the well-known Cauchy-Davenport theorem~\cite{Cauchy,Davenport} yields that every sum-free set has size at most $\lfloor (p+1)/3 \rfloor$. The sets attaining this bound were explicitly characterized in the late sixties by Yap~\cite{Yap68,Yap70}, Diananda and Yap~\cite{Diananda69}, and Rhemtulla and Street~\cite{RhemtullaStreet70}.

\begin{theorem}[\cite{Yap68,Diananda69,Yap70,RhemtullaStreet70}]\label{thm:YapIntro}
Let $p$ be a prime, and let $S \subseteq \Z_p$ be a sum-free set of maximum size.
\begin{enumerate}
  \item If $p=3k+1$ for an integer $k$ then $|S|=k$ and $S$ is equal, up to an automorphism, to one of the following sets:
  \begin{eqnarray}\label{eq:sets_Yap}
  [k+1,2k],~~~~[k,2k-1],~~\mbox{ and for $k \geq 4$,~~}\{k\} \cup [k+2,2k-1] \cup \{2k+1\}.
  \end{eqnarray}
  \item If $p=3k+2$ for an integer $k$ then $|S|=k+1$ and $S$ is equal, up to an automorphism, to the set $[k+1,2k+1]$.
\end{enumerate}
\end{theorem}
\noindent
Note that for $p=3k+2$ the sum-free subsets of $\Z_p$ of maximum size are symmetric and complete, whereas for $p=3k+1$ only sets of the third type in~\eqref{eq:sets_Yap} have these properties.
For a general finite abelian group $G$, the largest size of a sum-free set was determined in 2005 by Green and Ruzsa~\cite{GreenRuzsa05}.

Another question of interest is that of counting sum-free sets.
Cameron and Erd{\H{o}}s conjectured in~\cite{CameronErdos90} that the number of sum-free subsets of the integers $[1,n]$ is $O(2^{n/2})$. This conjecture was confirmed independently by Green~\cite{Green04} and by Sapozhenko~\cite{Sapozhenko03} (see also~\cite{AlonRefine14}), improving on Alon~\cite{Alon91}, Calkin~\cite{Calkin90}, and Erd{\H{o}}s and Granville (see~\cite{CameronErdos90}).
Counting sum-free sets was also considered for finite abelian groups.
For the cyclic group $\Z_p$ with $p$ prime, an essentially tight upper bound of $2^{p(1/3+o(1))}$ was proved by Green and Ruzsa~\cite{GreenRuzsa04}, who later extended their result to general finite abelian groups~\cite{GreenRuzsa05}.
The task of counting sum-free sets that are maximal with respect to inclusion was addressed in~\cite{BaloghLST15b}.

\subsection{Our Contribution}

In this paper we present families of complete sum-free subsets of various sizes of the cyclic group $\Z_n$ of order $n$.
Such sets have found several applications in the literature (see, e.g.,~\cite{Cameron84,HansonS84,Cameron87,Cameron87b,CalkinC98} and Section~\ref{sec:motivation}).

The first family, studied in Section~\ref{sec:large}, consists of subsets of $\Z_n$ whose size is linear in $n$.
These sets include a central symmetric interval and additional elements located symmetrically on its left and right.
Specifically, for $s \in [\frac{n+3}{4},\frac{n-1}{3}]$ and a set of integers $T \subseteq [0,2t-1]$, where $t$ is an appropriately chosen parameter, we consider the set $S_T \subseteq \Z_n$ defined as
\[ S_T = [n-2s+1,2s-1] \cup \pm (s+T)\]
(see Definition~\ref{def:S_T} and Figure~\ref{fig:S_T}).
We prove sufficient and necessary conditions on the set $T$ for which the set $S_T$ is sum-free and complete (see Lemmas~\ref{lemma:sum-free} and~\ref{lemma:completeness}) and refer to the sets $T$ of size $t$ satisfying these conditions as {\em $t$-special} (see Section~\ref{sec:special}).
We use the analysis of the sets $S_T$ to establish a variety of results, described below.

Firstly, we provide a full characterization of the `large' symmetric complete sum-free subsets of $\Z_p$ for primes $p$. Namely, we show that for every $s \geq (\frac{1}{3}-c) \cdot p$ where $c>0$ is a universal constant and $p$ is a sufficiently large prime, the sets $S_T \subseteq \Z_p$ of size $s$ for $t$-special sets $T$ are, up to automorphisms, all the symmetric complete sum-free subsets of $\Z_p$ of size $s$ (see Theorem~\ref{thm:Final_Char}).
This significantly generalizes the characterization of the symmetric complete sum-free sets of maximum size derived from Theorem~\ref{thm:YapIntro}. Our proof makes use of a result due to Deshouillers and Lev~\cite{DLev08} on the structure of large sum-free subsets of $\Z_p$ (see Theorem~\ref{thm:DLev}; see also~\cite{Lev06,DeshouillersF06}).

Secondly, we consider counting aspects of our characterization result. It is shown that counting the symmetric complete sum-free subsets of $\Z_p$ of a given size reduces to counting $t$-special sets for a certain value of $t$. In the following theorem, $g(t)$ stands for the number of $t$-special sets (see Section~\ref{sec:counting_large}).

\begin{theorem}\label{thm:counting_complete}
There exists a constant $c>0$ such that for every sufficiently large prime $p$ and every $1 \leq r \leq c \cdot p$ the following holds.
\begin{enumerate}
  \item If $p=3k+1$ for an integer $k$ then the number of symmetric complete sum-free subsets of $\Z_p$ of size $k-2r$ is $\frac{p-1}{2} \cdot g(3r+1)$.
  \item If $p=3k+2$ for an integer $k$ then the number of symmetric complete sum-free subsets of $\Z_p$ of size $k-2r+1$ is $\frac{p-1}{2} \cdot g(3r)$.
\end{enumerate}
\end{theorem}

Lastly, we observe that there is an abundance of $t$-special sets and derive the following.
\begin{theorem}\label{thm:CountIntro}
There exists a constant $c>0$ such that for every sufficiently large $n$ there exist at least $2^{cn}$ symmetric complete sum-free subsets of $\Z_n$.
\end{theorem}

In Section~\ref{sec:small}, we present for a general $n$ a family of symmetric complete sum-free subsets of $\Z_n$ of various sizes.
The construction combines sets $\pm A, \pm B, C$, where $A$ is an interval, $B$ is an arithmetic progression, and $C$ is a symmetric interval (see Figure~\ref{fig:the_set}).
A key property of these sets, used to prove completeness and sum-freeness, is that the sums of the elements from $A$ with the elements from $B$ precisely fill the gaps between the elements of $-B$.

We use our construction to prove that the relative sizes of the symmetric complete sum-free subsets of finite cyclic groups are dense in $[0,\frac{1}{3}]$.
This answers an open question of Cameron~\cite{CameronBlog10} related to a natural probability measure on sum-free sets of positive integers (see~\cite{Cameron87b,Cameron87,Calkin98,CalkinC98} and Section~\ref{sec:motivation}).

\begin{theorem}\label{thm:DenseIntro}
For every $0 \leq \alpha \leq \frac{1}{3}$ and $\eps >0$, for every sufficiently large integer $n$ there exists a symmetric complete sum-free set $S \subseteq \Z_n$ whose size satisfies
\[ \alpha -\eps \leq \frac{|S|}{n} \leq \alpha+\eps.\]
\end{theorem}

We also obtain, for a general $n$, a symmetric complete sum-free subset of $\Z_n$ of size proportional to $\sqrt{n}$. This matches, up to a multiplicative constant, the smallest possible size of such a set, and extends a result of Hanson and Seyffarth~\cite{HansonS84} that holds for cyclic groups $\Z_n$ where $n = m^2+5m+2$ for an integer $m$.

\begin{theorem}\label{thm:SmallIntro}
There exists a constant $c>0$ such that for every sufficiently large integer $n$ there exists a symmetric complete sum-free subset of $\Z_n$ of size at most $c \cdot \sqrt{n}$.
\end{theorem}

\subsection{Applications}\label{sec:motivation}

We gather here several applications of our results.

\paragraph{Random sum-free sets.}

Cameron introduced in~\cite{Cameron87b} a probability measure on the sum-free subsets of the positive integers.
This measure is defined via a process that constructs a random sum-free set $R$ as follows:
Go over the positive integers in turn, and join to $R$, independently with probability $1/2$, every element that cannot be expressed as a sum of two elements that are already in $R$.

The question of identifying sets $M$ for which the probability that $M$ contains a random sum-free set is a constant bounded away from $0$ was addressed in~\cite{Cameron87b,Cameron87,Calkin98,CalkinC98}. For example, it was shown in~\cite{Cameron87b} that the probability of a random sum-free set to consist entirely of odd numbers is about $0.218$ (the exact constant is unknown). Conditioned on this event, the density\footnote{The density of a set of positive integers is the limit of the proportion of the numbers in $[1,n]$ that belong to the set, as $n$ tends to infinity. We assume in this discussion that the density of a random sum-free set exists almost surely and refer the interested reader to~\cite{CameronBlog10} for more details.} of a random sum-free set is almost surely $1/4$. The intuitive reason is that if all the elements of the set are odd, then every odd number is unconstrained in the process and is thus joined to the set independently with probability $1/2$.

For a subset $S$ of $\Z_n$, let $M_S$ be the set of all positive integers that are equal modulo $n$ to elements from $S$. Cameron proved in~\cite{Cameron87b} that for every complete sum-free subset $S$ of $\Z_n$ the probability that a random sum-free set is contained in $M_S$ is nonzero (see also an extension by Calkin~\cite{Calkin98}). Conditioned on this event, it can be seen that the density of a random sum-free set is almost surely $|S|/(2n)$. Cameron asked in~\cite{CameronBlog10} (see also~\cite{Cameron87}) whether the set \[\Big \{ \frac{|S|}{2n} ~\Big | ~\mbox{$S$ is a complete sum-free subset of $\Z_n$,~$n \geq 1$} \Big \}\]
is dense in the interval $[0,\frac{1}{6}]$. Our Theorem~\ref{thm:DenseIntro} answers this question in the affirmative.

\paragraph{Regular triangle-free graphs with diameter $2$.}
A nice application of symmetric complete sum-free sets of `small' size comes from the following graph theory question: For an integer $n$, what is the smallest $d$ for which there exists an $n$ vertex $d$-regular triangle-free graph of diameter $2$?
This question was studied by Hanson and Seyffarth~\cite{HansonS84}, who were motivated by a problem of Erd{\H{o}}s, Hajnal, and Moon~\cite{ErdosHM64} on $k$-saturated graphs.
It was observed in~\cite{HansonS84} that if $S$ is a symmetric complete sum-free subset of an abelian group $G$ then the Cayley graph associated with $G$ and $S$ is an $n$ vertex $d$-regular triangle-free graph of diameter $2$ for $n=|G|$ and $d=|S|$.
The completeness of $S$ easily yields a lower bound of $|S| \geq \sqrt{2|G|}-O(1)$.
As already mentioned, our Theorem~\ref{thm:SmallIntro} extends a result of~\cite{HansonS84} and shows that this lower bound is attained, up to a multiplicative constant, for every cyclic group $\Z_n$ (see~\cite{CP92} for a construction for a non-cyclic group).

\paragraph{Dioid partitions of groups.}
The motivation for the current work has arrived from the concept of dioid partitions of groups introduced by the authors in~\cite{HavivL17}.
For a group $G$, written here in multiplicative notation, a partition $\Pi$ of $G$ is a {\em dioid partition} if it satisfies the following properties:
\begin{enumerate}
  \item For every $\pi_1, \pi_2 \in \Pi$, $\pi_1 \pi_2$ is a union of parts of $\Pi$, where $\pi_1 \pi_2$ is a setwise product.
  \item There exists $e \in \Pi$ satisfying $e \pi = \pi e =  \pi$ for every $\pi \in \Pi$.
  \item For every $\pi \in \Pi$, $\pi^{-1} \in \Pi$, where $\pi^{-1} = \{g^{-1} \mid g \in \pi\}$.
\end{enumerate}
\noindent
It is shown in~\cite{HavivL17} that dioid partitions of a group $G$ naturally define algebraic structures known as dioids (see, e.g.,~\cite{GondranMinouxDioidBook2010}).
Dioid partitions of $G$ are provided, for example, by the set of all conjugacy classes of $G$ and by the set of all double cosets of any subgroup of $G$.
For a cyclic group of prime order these examples yield only the two `trivial' partitions
$\{ G\} $ and $\{ \{ g\} \mid g\in G \}$.
However, for a prime $p\geq 5$, every symmetric complete sum-free subset $S$ of $\Z_{p}$ yields the non-trivial $3$-part dioid partition
$\{\{0\},S,(S+S) \setminus \{0\}\}$ (see~\cite{HavivL17}). Our results in the current paper illuminate some of the structural and counting aspects of these partitions.

\section{Preliminaries}

\paragraph{Symmetry, completeness, and sum-freeness.}
For an additive group $G$ and sets $A,B \subseteq G$ we denote $A+B = \{x+y \mid x \in A,~y \in B\}$, $-A = \{-x \mid x \in A\}$, and $A-B = A+(-B)$. We also denote $\pm A = A \cup (-A)$. The set $A$ is said to be {\em symmetric} if $A = -A$. It is {\em complete} if $G \subseteq A \cup (A+A)$, and it is {\em sum-free} if $A \cap (A+A) = \emptyset$.
Note that each of these three properties is preserved by any automorphism of $G$.
Notice that $A$ is complete and sum-free if and only if $A+A = G \setminus A$.

The following claim will be used to prove sum-freeness and completeness of sets.
\begin{claim}\label{claim:sumfree_short}
Let $G$ be an additive group, let $G_1 \subseteq G$ be a set satisfying $G = G_1 \cup (-G_1)$, and let $S \subseteq G$ be a symmetric set. Then,
\begin{enumerate}
  \item If every $x,y \in G_1 \cap S$ satisfy $x+y \notin S$ then $S$ is sum-free.
  \item If $G_1 \setminus S \subseteq S+S$ then $S$ is complete.
\end{enumerate}
\end{claim}

\begin{proof}
For the first item, let $x,y \in S$ and denote $z=x+y$. We prove that $z \notin S$. If $x,y \in G_1$ then by assumption $z \notin S$. If $-x,-y \in G_1$ then by assumption $-z \notin S$, which by symmetry of $S$ implies $z \notin S$. Otherwise, without loss of generality, $x,-y \in G_1$. Assume by contradiction that $z \in S$. If $z \in G_1$ then $x$, which equals the sum of $-y$ and $z$ is not in $S$, a contradiction. Otherwise, $-z \in G_1$ and thus $-y$, which equals the sum of $x$ and $-z$, is not in $S$, again a contradiction.

For the second item, let $z \in G \setminus S$. We prove that $z \in S+S$. If $z \in G_1$ then by assumption $z \in S+S$. Otherwise, $z \in -G_1$, thus $(-z) \in G_1 \setminus S$, implying that $-z \in S+S$, which by symmetry of $S$ implies that $z \in S+S$ as well.
\end{proof}

\paragraph{Cyclic groups.}
For an arbitrary integer $n \geq 1$, we consider the cyclic group $\Z_n = \{0,1,\ldots,n-1\}$, with addition modulo $n$. Every element $a \in \Z_n$ has infinitely many representatives in $\Z$, those of the form $a+m \cdot n$ for $m \in \Z$.
For integers $a \leq b$ we denote $[a,b] = \{ z \in \Z \mid a \leq z \leq b\}$. We use the same notation for intervals in $\Z_n$, viewing the elements of $[a,b]$ as representatives of elements of $\Z_n$.
Note that we use the same addition symbol for both integer addition and addition modulo $n$.
In case of possible ambiguity, we will explicitly mention what the notation means.
For a prime $p$, the group automorphisms of $\Z_p$ are realized as multiplication modulo $p$ by an element of $\Z_p \setminus \{0\}$.
For $d \in \Z_p \setminus \{0\}$ and $A \subseteq \Z_p$ the set $d \cdot A \subseteq \Z_p$ is said to be a {\em dilation} of $A$.

\paragraph{Large sum-free sets of cyclic groups of prime order.}
We state below a theorem of Deshouillers and Lev~\cite{DLev08} on the structure of large sum-free subsets of $\Z_p$ for a prime $p$.

\begin{theorem}[\cite{DLev08}]\label{thm:DLev}
For every sufficiently large prime $p$, every sum-free subset $S$ of $\Z_p$ of size $|S| \geq 0.318p$ is contained in a dilation of the interval $[|S|,p-|S|]$.
\end{theorem}
\noindent
Theorem~\ref{thm:DLev} improves previous results of Lev~\cite{Lev06} and Deshouillers and Freiman~\cite{DeshouillersF06}.
Dainiak and Sisask independently proved that the constant $0.318$ in the theorem cannot get below $0.25$ (see, e.g.,~\cite{SisaskThesis}).
Finding the smallest constant for which the statement holds is an interesting open question.

\section{Large Symmetric Complete Sum-free Sets}\label{sec:large}

In this section we present and analyze a construction of `large' symmetric complete sum-free subsets of the cyclic group $\Z_n$.

\subsection{The Sets $S_T$}
Consider the following definition.
\begin{definition}\label{def:S_T}
Let $n$ and $s$ be integers for which $t=(n-3s+1)/2$ is a positive integer and $n \leq 4s-3$.
For a set $T \subseteq [0,2t-1]$, let $S_T \subseteq \Z_n$ be the set defined as
\[S_T = [n-2s+1,2s-1] \cup \pm (s+T).\]
\end{definition}

\begin{figure}[h]
\begin{center}
\includegraphics[width=6in]{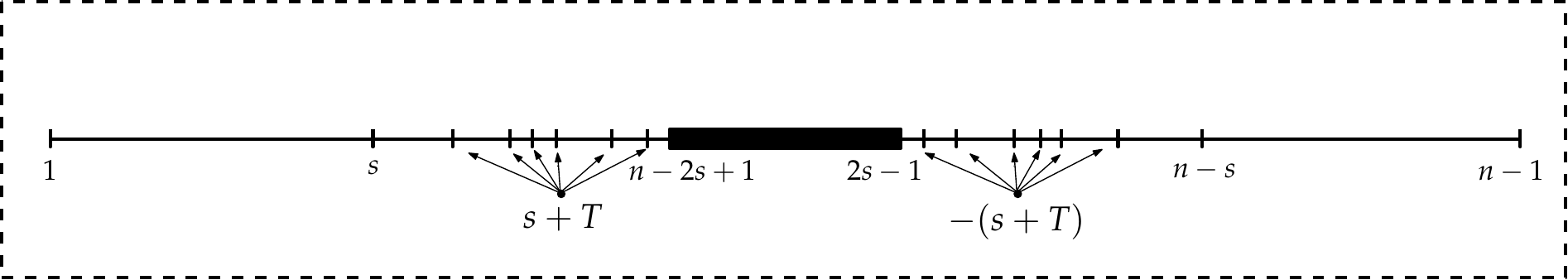}
\end{center}
\caption{An illustration of the elements of the set $S_T$.}
\label{fig:S_T}
\end{figure}

A few remarks are in order.
The set $S_T \subseteq \Z_n$ defined above depends on a parameter $s$ and a set $T$.
The definition requires the parameters $n$ and $s$ to satisfy
\[\frac{n+3}{4} \leq s \leq \frac{n-1}{3},\]
where the upper bound on $s$ follows from $t \geq 1$.
Notice that the interval $[n-2s+1,2s-1]$ included in $S_T$ is symmetric, thus the set $S_T$ is symmetric as well.
The set $s+T$ includes $|T|$ elements, which all belong to the interval $[s,s+2t-1]$. By the definition of $t$ we have $s+2t-1 = n-2s$, hence the elements of $s+T$ are located on the left of the interval $[n-2s+1,2s-1]$ (see Figure~\ref{fig:S_T}). By symmetry, the elements of $-(s+T)$ are located on its right. This implies that $S_T \subseteq [s,n-s]$ and that the size of $S_T$ is
\[ |S_T| = 4s-n-1+2|T|.\]
In the case of $|T|=t$, which will be of particular interest for us, we have $|S_T| = s$.

We consider now the question of what conditions on the set $T$ make the set $S_T$ a complete sum-free subset of $\Z_n$.
The following lemma gives a sufficient and necessary condition on $T$ for sum-freeness of $S_T$.

\begin{lemma}[Sum-freeness of $S_T$]\label{lemma:sum-free}
Let $n$ and $s$ be integers for which $t=(n-3s+1)/2$ is a positive integer and $n \leq 4s-3$, and let $T \subseteq [0,2t-1]$ be a set.
Then, the set $S_T \subseteq \Z_n$ is sum-free if and only if
\begin{eqnarray}\label{lemma:add1}
2t-1 \notin T+T+T,
\end{eqnarray}
where the addition in~\eqref{lemma:add1} is over the integers.
\end{lemma}

\begin{proof}
Denote $A = [n-2s+1,2s-1]$, and recall that $S_T = A \cup \pm (s+T) \subseteq [s,n-s]$.
Let us consider the elements of $S_T+S_T$.
Observe that
\begin{eqnarray}\label{A+A}
A+A = [2n-4s+2,4s-2] = [0,4s-n-2] \cup [2n-4s+2,n-1],
\end{eqnarray}
where for the second equality we have used $n \leq 4s-3$.
By $t \geq 1$ we have
\[4s-n-2 = s - (n-3s+1) -1= s-2t-1 <s,\]
hence, using the symmetry of $A$ and $S_T$, it follows that $A+A$ is disjoint from $S_T$.

Denote $B = [s,s+2t-1]$, and notice that $s+T \subseteq B$. The set
\[A+B = [n-s+1, 3s+2t-2] = [n-s+1, n-1] = -[1,s-1]\]
is disjoint from $S_T$, because $S_T$ is symmetric and has no intersection with $[1,s-1]$.
This implies that the sum of $A$ and $s+T$ is disjoint from $S_T$ as well.

It remains to consider the sums of pairs of elements in $\pm (s+T)$. The sum of $s+T$ and $-(s+T)$ is contained in $\pm [0,2t-1]$, which is disjoint from $S_T$
because \[2t-1<2t = n-3s+1<s.\]

We turn to show that the condition $2t-1 \notin T+T+T$ holds if and only if the sum of every two elements in $s+T$ does not belong to $S_T$. This would imply that this condition is equivalent to the sum-freeness of $S_T$ and thus would complete the proof of the lemma.

First, assume that $2t-1 \notin T+T+T$, and let $s+\ell_1$ and $s+\ell_2$ be two elements of $s+T$ where $\ell_1, \ell_2 \in T$.
Set $\ell_3 = 2t-1-(\ell_1+\ell_2)$, and note that the assumption $2t-1 \notin T+T+T$ implies that $\ell_3 \notin T$.
The sum of $s+\ell_1$ and $s+\ell_2$ is
\begin{eqnarray}\label{eq:sum_of_2_elem}
2s+\ell_1+\ell_2 = 2s+(2t-1-\ell_3) = -(n-2s-2t+1+\ell_3) = -(s+\ell_3),
\end{eqnarray}
so by symmetry of $S_T$ it suffices to show that $s+\ell_3 \notin S_T$.
If $\ell_3 \geq 0$ then $\ell_3 \in [0,2t-1]$, implying $s+\ell_3 \in B$. However, $B \cap S_T = s+T$, so by $\ell_3 \notin T$ it follows that $s+\ell_3$ is not in $S_T$, as required.
Otherwise, $\ell_3 < 0$, and we claim that in this case $s+\ell_3 \in [0,s-1]$. Indeed, $s+\ell_3 < s$, and by $\ell_3 \geq -(2t-1)$ we obtain that
\[s+ \ell_3 \geq s-(2t-1) = 4s-n > 0.\]
Since the interval $[0,s-1]$ is disjoint from $S_T$, it follows that $s+\ell_3$ is not in $S_T$, and we are done.

For the other direction, assume that $2t-1 \in T+T+T$, that is, there exist $\ell_1, \ell_2, \ell_3 \in T$ such that $\ell_3 = 2t-1-(\ell_1+\ell_2)$. By the definition of $S_T$, the elements $s+\ell_1$, $s+\ell_2$ and $s+\ell_3$ are in $S_T$. As in~\eqref{eq:sum_of_2_elem}, the sum of the first two is $-(s+\ell_3)$,
which by symmetry of $S_T$ does belong to $S_T$. Hence, the sum of two elements of $s+T$ belongs to $S_T$, as required.
\end{proof}

We now consider the completeness of $S_T$, that is, the property that every element of $\Z_n$ belongs to either $S_T$ or to $S_T+S_T$ (or both). The following lemma provides a sufficient and necessary condition on the set $T$, assuming that $s$ is sufficiently large compared to $n$.
We use here the notation $\min(T)$ for the smallest integer in $T$.

\begin{lemma}[Completeness of $S_T$]\label{lemma:completeness}
Let $n$ and $s$ be integers for which $t=(n-3s+1)/2$ is a positive integer and $n \leq 7s/2-1$, and let $T \subseteq [0,2t-1]$ be a nonempty set.
Then, the set $S_T \subseteq \Z_n$ is complete if and only if
\begin{eqnarray}\label{lemma:add2}
[0,2t-1+\min(T)] \setminus (2t-1-T) \subseteq T+T,
\end{eqnarray}
where the addition in~\eqref{lemma:add2} is over the integers.
\end{lemma}

\begin{proof}
Denote $m = \min(T)$ and $A = [n-2s+1,2s-1]$ as before.
By~\eqref{A+A}, we have
\begin{eqnarray}\label{eq:lem_com1}
[0,4s-n-2] \subseteq A+A \subseteq S_T+S_T.
\end{eqnarray}
In addition, we have $s + m \in s+T \subseteq S_T$ and
\[(s+m)+A = [n-s+m+1,3s+m-1] = -[n-3s-m+1,s-m-1].\]
By symmetry of $S_T$ this implies that
\begin{eqnarray}\label{eq:lem_com2}
[n-3s-m+1,s-m-1] \subseteq S_T+S_T.
\end{eqnarray}
Observe that
\begin{eqnarray}\label{eq:lem_com3}
n-3s-m+1 \leq n-3s+1 \leq 4s-n-1,
\end{eqnarray}
where the second inequality follows from the assumption $n \leq 7s/2-1$.
Combining~\eqref{eq:lem_com1},~\eqref{eq:lem_com2} and~\eqref{eq:lem_com3}, we obtain that
\[ [0,s-m-1] \subseteq S_T + S_T.\]
Since the interval $A$ starts at $n-2s+1$, using the symmetry of $S_T$, it follows that
\[\Z_n \subseteq S_T \cup (S_T + S_T)
 \mbox{ ~~if and only if~~ } [s-m,n-2s] \subseteq S_T \cup (S_T + S_T).\]

The interval $[s-m,n-2s]$ does not intersect the following sets:
\begin{itemize}
  \item $A$, because $n-2s < n-2s+1$;
  \item $-(s+T)$, because $-(s+T) \subseteq [n-s-(2t-1),n-s]$ and
  \[n-2s < 2s = n-s-(2t-1);\]
  \item $A+A$, by~\eqref{A+A} and the fact that $4s-n-2 = s-2t-1 < s-(2t-1) \leq s-m$;
  \item $(s+T) + (s+T)$, because this set is contained in $[2s, 2s+4t-2]$, and $n-2s < 2s$;
  \item $(s+T) - (s+T)$, because this set is contained in $\pm [0, 2t-m-1]$, and \[2t-m-1 = n-3s-m < s-m;\]
  \item $(s+T)+A$, because this set is contained in \[s+[m,2t-1]+A = [s+m,s+2t-1]+A = [n-s+m+1,3s+2t-2],\] and $n-2s < n-s+m+1$;
  \item $-(s+T)+A$, because this set is contained in \[-(s+[m,2t-1]+A) = [n-3s-2t+2, s-m-1],\] and $s-m-1 < s-m$.
\end{itemize}

Therefore, the only way for the elements of $[s-m,n-2s]$ to be in $S_T \cup (S_T + S_T)$ is to belong to either $s+T$ or to
\[ -(s+T)-(s+T) = -2s-(T+T) = n-2s-(T+T).\]
The condition \[[s-m,n-2s] \subseteq (s+T) \cup (n-2s-(T+T))\]
is equivalent to
\[ [0,n-3s+m] \subseteq (T+m) \cup (n-3s+m-(T+T)),\]
which by the definition of $t$ is equivalent to
\[ [0,2t-1+m] \subseteq (T+m) \cup (2t-1+m-(T+T)),\]
and the latter, using the fact that $[0,a] = a-[0,a]$, is equivalent to the condition in the lemma, so we are done.
\end{proof}

\subsection{Special Sets}\label{sec:special}

We now turn to identify the sets $T$ for which $S_T$ is a complete sum-free subset of $\Z_n$ of size $s$ (see Definition~\ref{def:S_T}).
Lemmas~\ref{lemma:sum-free} and~\ref{lemma:completeness} supply the required conditions on $T$ for sum-freeness and completeness of $S_T$.
Recall that the size of $S_T$ is $s$ whenever $|T|=t$.
This leads us to the notion of {\em $t$-special sets} given below.

\begin{definition}\label{def:special}
For an integer $t \geq 1$ we say that a set $T \subseteq [0,2t-1]$ is {\em $t$-special} if it satisfies
\begin{enumerate}
  \item $|T| = t$,
  \item $2t-1 \notin T+T+T$, and
  \item $[0,2t-1+\min(T)] \setminus (2t-1-T) \subseteq T+T$.
\end{enumerate}
Note that the addition is over the integers.
\end{definition}

The following simple claim shows that in case that $0 \in T$ the first and second conditions in Definition~\ref{def:special} imply the third.
\begin{claim}\label{claim:0in}
For every integer $t \geq 1$ and a set $T \subseteq [0,2t-1]$ satisfying $0 \in T$,
if $|T|=t$ and $2t-1 \notin T+T+T$ then $T$ is $t$-special.
\end{claim}

\begin{proof}
Assume that $0 \in T$, $|T|=t$ and $2t-1 \notin T+T+T$. We claim that for every $\ell \in [0,t-1]$ exactly one of the elements $\ell$ and $2t-1-\ell$ belongs to $T$. Indeed, the elements $\ell$ and $2t-1-\ell$ cannot both belong to $T$ as their sum, together with $0 \in T$, is $2t-1$, which does not belong to $T+T+T$.
In addition, if $T$ does not contain one of $\ell$ and $2t-1-\ell$ for some $\ell \in [0,t-1]$ then its size cannot reach $t$. It follows that every $\ell \in [0,2t-1]$ for which $2t-1-\ell \notin T$ satisfies $\ell = \ell+0 \in T+T$. This implies that $[0,2t-1] \setminus (2t-1-T) \subseteq T+T$, hence $T$ is $t$-special.
\end{proof}

\begin{remark}
It is easy to notice that for every $t \geq 1$ there exist $t$-special sets. For example, the set $\{0\} \cup [t,2t-2]$ and the set of even elements in $[0,2t-1]$ are $t$-special, as is easy to verify using Claim~\ref{claim:0in}.
\end{remark}

Combining Lemmas~\ref{lemma:sum-free} and~\ref{lemma:completeness} with the fact that $|S_T|=s$ if and only if $|T|=t$, we obtain the following result.

\begin{theorem}\label{thm:complete}
Let $n$ and $s$ be integers for which $t=(n-3s+1)/2$ is a positive integer and $n \leq 7s/2-1$, and let $T \subseteq [0,2t-1]$ be a set.
Then, $T$ is $t$-special if and only if $S_T$ is a complete sum-free subset of $\Z_n$ of size $s$.
\end{theorem}

\subsection{Characterizing the Large Symmetric Complete Sum-free Sets in $\Z_p$}

Consider the cyclic group $\Z_p$ for a prime $p$. We provide a characterization, based on Definitions~\ref{def:S_T} and~\ref{def:special}, of the symmetric complete sum-free subsets of $\Z_p$ of a given size $s$ for large values of $s$. By `large' we mean that $s \geq (\frac{1}{3}-c) \cdot p$ for a universal constant $c>0$ (the proof will give $s \geq 0.318p$). Recall that the maximum possible size of such a set is $\lfloor (p+1)/3 \rfloor$, making the restriction on $s$ quite natural.

\begin{theorem}\label{thm:Final_Char}
There exists a constant $c>0$ such that for every sufficiently large prime $p$ and an integer $s$, for which $t=(p-3s+1)/2$ is a positive integer satisfying $t \leq c \cdot p$, the following holds.
The symmetric complete sum-free subsets of $\Z_p$ of size $s$ are precisely all the dilations of the sets $S_T$ for $t$-special sets $T \subseteq [0,2t-1]$.
\end{theorem}

Note that Theorem~\ref{thm:Final_Char} characterizes the symmetric complete sum-free subsets of $\Z_p$ of size $s$ for every possible size $s$ of such a set satisfying
\begin{eqnarray}\label{eq:cond_s}
\frac{p(1-2c)+1}{3} \leq s \leq \frac{p-1}{3},
\end{eqnarray}
where $c>0$ is the constant in Theorem~\ref{thm:Final_Char}.
Indeed, the size of every symmetric sum-free subset of $\Z_p$, for a prime $p>2$, is even.
For an even $s$ satisfying~\eqref{eq:cond_s} the number $t=(p-3s+1)/2$ is an integer with $1 \leq t \leq c \cdot p$.

The fact that the sets $S_T$ for $t$-special sets $T \subseteq [0,2t-1]$ are symmetric complete sum-free subsets of $\Z_p$ already follows from Theorem~\ref{thm:complete}. This immediately implies that their dilations have these properties as well. Hence, to complete the proof of Theorem~\ref{thm:Final_Char} it suffices to show that these are the only sets with those properties. This is the goal of the following lemma.

\begin{lemma}\label{lemma:Char}
There exists a constant $c>0$ such that for every sufficiently large prime $p$ and an integer $s$, for which $t=(p-3s+1)/2$ is a positive integer satisfying $t \leq c \cdot p$, the following holds.
Every symmetric complete sum-free subset of $\Z_p$ of size $s$ is a dilation of a set $S_T$ where $T \subseteq [0,2t-1]$ is $t$-special.
\end{lemma}

\begin{proof}
Let $S \subseteq \Z_p$ be a symmetric complete sum-free set of size $|S|=s$.
Define $c=0.023$ (the constant satisfying $\frac{1}{3} \cdot (1-2c)=0.318$).
By the definition of $t$, the assumption $t \leq c \cdot p$ implies that
\[s = \frac{p-2t+1}{3} \geq \frac{p(1-2c)+1}{3} \geq 0.318p.\]
Theorem~\ref{thm:DLev} implies that the sum-free set $S$ is contained in a dilation of the interval $C=[s,p-s]$. As the statement of our lemma is invariant under dilations, we can assume that $S \subseteq C$.
Observe that \[C+C = [2s,2p-2s] = [0,p-2s] \cup [2s,p-1].\]
By $S \subseteq C$ we have $S+S \subseteq C+C$, so the completeness of $S$ implies that $\Z_p \setminus (C+C) \subseteq S$, that is, \[[p-2s+1,2s-1] \subseteq S.\]
Thus we have identified $4s-p-1$ elements of $S$. The remaining $s-(4s-p-1) = 2t$ elements of $S$ belong to
\[[s,p-2s] \cup [2s,p-s],\]
so by symmetry they are fully defined by the $t$ elements of $S$ in $[s,p-2s]$. Denote the elements of $S \cap [s,p-2s]$ by $s+T$ where $T \subseteq [0,2t-1]$. This yields that $S$ is equal to a set $S_T$ as given in Definition~\ref{def:S_T} where $|T|=t$. Applying Theorem~\ref{thm:complete} we obtain that $T$ is $t$-special, as desired.
\end{proof}

\subsubsection{Symmetric Complete Sum-free Sets of Almost Maximum Size}

For a prime $p$, Theorem~\ref{thm:Final_Char} essentially reduces the task of computing the symmetric complete sum-free subsets of $\Z_p$ of a given size to that of computing the corresponding $t$-special sets. We demonstrate this reduction below and derive an explicit characterization of the symmetric complete sum-free subsets of $\Z_p$ of `almost maximum size', that is, size smaller by $2$ than the maximum size of such a subset.

For a prime $p=3k+1$ set $s=k-2$, and notice that $t = (p-3s+1)/2 = 4$.
It is easy to verify that the $4$-special sets $T \subseteq [0,7]$ are
\[\{0,4,5,6\},~~\{0,2,4,6\},~~\{0,3,5,6\},~~\{1,2,6,7\}.\]
For every such $T$ this gives us the following symmetric complete sum-free subset of $\Z_p$:
\[((k-2)+T) \cup [k+6,2k-5] \cup ((2k+3)-T).\]
Assuming that $p$ is sufficiently large, the dilations of these sets are the only symmetric complete sum-free subsets of $\Z_p$ of size $k-2$.

Similarly, for a prime $p=3k+2$ set $s=k-1$, and notice that $t = (p-3s+1)/2 = 3$.
The $3$-special sets $T \subseteq [0,5]$ are
\[\{0,2,4\},~~\{0,3,4\}.\]
For every such $T$ this gives us the following symmetric complete sum-free subset of $\Z_p$:
\[((k-1)+T) \cup [k+5,2k-3] \cup ((2k+3)-T).\]
Again, assuming that $p$ is sufficiently large, the dilations of these sets are the only symmetric complete sum-free subsets of $\Z_p$ of size $k-1$.

\subsection{Counting the Large Symmetric Complete Sum-free Sets}\label{sec:counting_large}

We now consider the question of estimating the number of symmetric complete sum-free subsets of $\Z_p$ of a given size $s$.
For every $t \geq 1$ denote the number of $t$-special sets by $g(t)$.
We prove below Theorem~\ref{thm:counting_complete} that reduces counting the symmetric complete sum-free subsets of $\Z_p$ of size $s$ to computing $g(t)$ for an appropriately chosen value of $t$.

\begin{proof}[ of Theorem~\ref{thm:counting_complete}]
Let $p=3k+1$ be a sufficiently large prime. For an integer $r \geq 1$, denote $s=k-2r$ and $t = (p-3s+1)/2 = 3r+1$. Assuming that $r \leq c \cdot p$ for a sufficiently small $c$, Theorem~\ref{thm:Final_Char} implies that the symmetric complete sum-free subsets of $\Z_p$ of size $s$ are the dilations of the sets $S_T$ for the $t$-special sets $T$.

We first prove that for every $2 \leq d \leq p-2$ and every $t$-special sets $T$ and $T'$ (including the case $T=T'$) it holds that $S_T \neq d \cdot S_{T'}$.
By the symmetry of $S_T$ we may assume $2 \leq d \leq (p-1)/2$.
Consider the interval $D = [-(s-1),s-1]$, and notice that
\[|D| = 2s-1 = 2(k-2r)-1 \geq 2(k-2cp)-1 \geq \Big ( \frac{2}{3}-5c \Big)p \geq d,\]
where the last inequality holds for a sufficiently small $c$.
Consider the set
\[E = d \cdot [p-2s+1,2s-1] \subseteq d \cdot S_{T'},\]
whose elements form an arithmetic progression of difference $d$.
As $D$ is disjoint from $S_T$, it suffices to show that $D$ intersects $E$.
Assume by contradiction that $D$ does not include any element of $E$.
By $|D| \geq d$ it follows that $D$ does not include any element located between two consecutive elements of $E$ as well.
This implies that $[d(p-2s+1),d(2s-1)]$ is an interval disjoint from $D$.
As its size is $d(4s-p-2)+1$, it follows that the total number of elements in the group is at least
\begin{eqnarray*}
d(4s-p-2)+1 + |D| &\geq& 2(4s-p-2)+1 + (2s-1) = 10s-2p-4
\\ &=& 10(k-2r)-2(3k+1)-4 = 4k-20r-6 \geq \Big(\frac{4}{3}-21c \Big )p > p,
\end{eqnarray*}
where for the last two inequalities we again assume that $c$ is sufficiently small. This clearly gives a contradiction.

Now, for every $t$-special set $T$ we have to count the $(p-1)/2$ dilations of $S_T$, hence the total number of symmetric complete sum-free subsets of $\Z_p$ of size $s$ is $(p-1)/2$ times the number of $(3r+1)$-special sets, as required.

The proof of the second item is essentially identical. Here, for $p=3k+2$ and $s=k-2r+1$ we have $t = (p-3s+1)/2 = 3r$.

\end{proof}

\subsubsection{Counting Special Sets}

Motivated by Theorem~\ref{thm:counting_complete}, we would like to understand the behavior of the number $g(t)$ of $t$-special sets as a function of $t$.
While $g(t)$ is trivially bounded from above by ${{2t} \choose t} < 2^{2t}$, the following claim bounds $g(t)$ from below and shows that it grows exponentially in $t$.

\begin{claim}\label{claim:number_T}
For every integer $t \geq 1$, $g(t) \geq 2^{\lfloor t/3 \rfloor}$.
\end{claim}

\begin{proof}
For every set $I \subseteq [\lceil 2t/3 \rceil, t-1]$ consider the set $T_I$ defined as
\[T_I = \{0\} \cup I \cup \{ 2t-1-i \mid \lceil 2t/3 \rceil \leq i \leq t-1,~ i \notin I\} \cup [2t-\lceil 2t/3 \rceil,2t-2]. \]
Equivalently, the set $T_I$ consists of $\{0\} \cup I$ and every element $2t-1-i$ for $i \in [0,t-1]$ with $i \notin \{0\} \cup I$.

As there are $2^{\lfloor t/3 \rfloor}$ possibilities for the set $I$, each of which gives a different $T_I$, it suffices to verify that those sets $T_I$ are $t$-special.
The set $T_I$ includes $0$ and has size $t$. By Claim~\ref{claim:0in}, it suffices to prove that no three elements of $T_I$ have sum $2t-1$. If the three elements are nonzero then their sum is at least $2t$. If exactly one of them is zero, then the sum of the other two cannot equal to $2t-1$ because $T_I$ includes exactly one of $i$ and $2t-1-i$ for every $i \in [0,t-1]$.
Finally, if at least two of the elements are zeros then their sum with the third is at most $2t-2$, and we are done.
\end{proof}

It is easy to derive now that the number of symmetric complete sum-free subsets of $\Z_n$ is exponential in $n$, confirming Theorem~\ref{thm:CountIntro}.

\begin{proof}[ of Theorem~\ref{thm:CountIntro}]
For a sufficiently large $n$, let $s$ be the smallest integer for which $n \leq 7s/2-1$ and $t = (n-3s+1)/2$ is a positive integer. Notice that for this choice we have $\lfloor t/3 \rfloor \geq c \cdot n$ for some universal constant $c$ (any $c<1/42$ would suffice).
By Theorem~\ref{thm:complete}, for every $t$-special set $T \subseteq [0,2t-1]$ the set $S_T$ is a symmetric complete sum-free subset of $\Z_n$ . By Claim~\ref{claim:number_T} there are at least $2^{\lfloor t/3 \rfloor} \geq 2^{cn}$ $t$-special sets $T$, and since each of them defines a different $S_T$, we are done.
\end{proof}

\section{Symmetric Complete Sum-free Sets of Various Sizes}\label{sec:small}

In this section we present a construction of symmetric complete sum-free subsets of the cyclic group $\Z_n$ that yields such sets with various sizes, confirming Theorems~\ref{thm:DenseIntro} and~\ref{thm:SmallIntro}.

Let $t \geq 1$, $d \geq 2$ and $k \geq 4$ be integers, and let $n$ be one of the following two integers
\begin{eqnarray}\label{eq:n_values}
4dk+6t-11,~~ 4dk+6t-14.
\end{eqnarray}
Note that every sufficiently large integer $n$ can be expressed as one of these two numbers for certain choices of the parameters.
Consider the following three subsets of $\Z_n$:
\begin{itemize}
  \item $A = [\frac{1}{2}(\lceil \frac{n}{2} \rceil+t+1), \frac{1}{2}(\lceil \frac{n}{2} \rceil+t+1) + d-2]$,
  \item $B = \Big\{ \frac{1}{2} (\lceil \frac{n}{2} \rceil+t+1 ) + 2d-2 + i \cdot d~ \Big|~ 0 \leq i \leq k-4  \Big\}$, and
  \item $C = [ \lfloor \frac{n}{2} \rfloor -t, \lceil \frac{n}{2} \rceil+t]$.
\end{itemize}
Define the set
\[ S^{(n)}_{t,d,k} = ( \pm A) \cup ( \pm B) \cup C.\]

\begin{figure}[h]
\begin{center}
\includegraphics[width=6in]{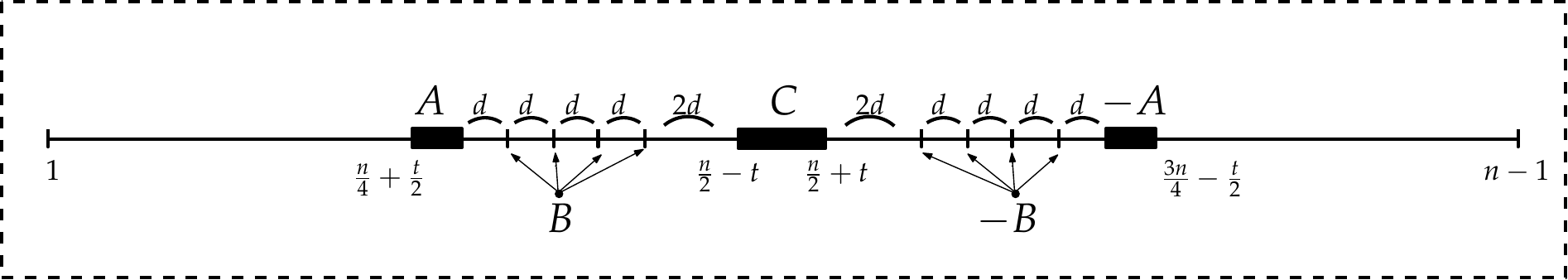}
\end{center}
\caption{An illustration of the elements of the set $S^{(n)}_{t,d,k}$. Some additive constants are omitted.}
\label{fig:the_set}
\end{figure}

Note that $A$ is an interval of length $d-1$ and that $B$ is an arithmetic progression of difference $d$ that includes $k-3$ elements, all located on the right of the elements of $A$ (see Figure~\ref{fig:the_set}). The set $C$ is a symmetric interval in $\Z_n$, whose size $|C|$ is $2t+1$ if $n$ is even and $2t+2$ if $n$ is odd. It follows from the definition that the set $S^{(n)}_{t,d,k}$ is also symmetric in $\Z_n$.

A simple calculation shows, for the two possible values of $n$ given in~\eqref{eq:n_values}, that the last element of $B$, denoted $b_l$, satisfies
\begin{eqnarray}\label{eq:b_l}
b_l = \Big \lfloor \frac{n}{2} \Big \rfloor - t -2d+2.
\end{eqnarray}
Indeed, for $n = 4dk+6t-11$ we have $ \lceil n/2 \rceil = 2dk+3t-5$ and $\lfloor n/2 \rfloor = 2dk+3t-6$, hence
\begin{eqnarray*}
b_l &=& \frac{1}{2} \Big ( \Big \lceil \frac{n}{2} \Big \rceil+t+1 \Big ) + 2d-2 + (k-4) \cdot d
\\ &=& \frac{1}{2}(2dk+4t-4)+2d-2+(k-4) \cdot d \\ &=& 2dk+2t-2d-4 = \Big \lfloor \frac{n}{2} \Big \rfloor -t-2d+2.
\end{eqnarray*}
For $n = 4dk+6t-14$ the equality is obtained similarly.
It follows that $b_l < \lfloor n/2 \rfloor - t$, hence the elements of $B$ are all located on the left of those of $C$ (see Figure~\ref{fig:the_set}). In particular, we obtain that the sets $\pm A,\pm B,C$, considered modulo $n$, are pairwise disjoint, implying that
\begin{eqnarray}\label{eq:size_S}
| S^{(n)}_{t,d,k} | = 2|A|+2|B|+|C| = 2(d+k-4)+|C|.
\end{eqnarray}
Observe that the sets $-A$ and $-B$ can be explicitly written as
\begin{itemize}
  \item $-A = [\frac{1}{2}(\lfloor \frac{3n}{2} \rfloor-t-1)-d+2, \frac{1}{2}(\lfloor \frac{3n}{2} \rfloor-t-1)]$ and
  \item $-B = \Big\{ \lceil \frac{n}{2} \rceil+t+2d-2 + i \cdot d~ \Big|~ 0 \leq i \leq k-4  \Big\}$,
\end{itemize}
where for $-B$ we have used~\eqref{eq:b_l}.

The main result of this section is the following.
\begin{theorem}\label{thm:complete_small}
For all integers $t \geq 1$, $d \geq 2$, $k \geq 4$ and $n \in \{4dk+6t-11,4dk+6t-14\}$ such that $|C| \geq d$ (that is, $d \leq 2t+1$ if $n$ is even and $d \leq 2t+2$ if $n$ is odd), the set $S^{(n)}_{t,d,k}$ is a symmetric complete sum-free subset of $\Z_n$.
\end{theorem}

In order to prove the theorem one has to show that every element of $\Z_n$ belongs to exactly one of $S^{(n)}_{t,d,k}$ and $S^{(n)}_{t,d,k} + S^{(n)}_{t,d,k}$. We start with the following claim that shows that the sum $A+B$ precisely fills the gaps between the elements of the arithmetic progression $-B$.

\begin{claim}\label{claim:-B,A+B}
The sets $-B$ and $A+B$ are disjoint and their union is
\[ \Big [  \Big \lceil \frac{n}{2} \Big \rceil + t+2d-2, \frac{1}{2} \Big ( \Big \lfloor \frac{3n}{2} \Big \rfloor -t-1 \Big)-d+1 \Big ]. \]
\end{claim}

\begin{proof}
The set $-B$ is an arithmetic progression of difference $d$ and length $k-3$ starting at $\lceil \frac{n}{2} \rceil + t+2d-2$.
The set $A$ is a $d-1$ length interval, so adding it to the arithmetic progression $B$ results in a pairwise disjoint union of $k-3$ shifts of $A$. The first element of $A+B$ is
\[  \frac{1}{2} \Big ( \Big \lceil \frac{n}{2} \Big \rceil +t+1 \Big )  + \Big( \frac{1}{2} \Big ( \Big \lceil \frac{n}{2} \Big \rceil +t+1 \Big )+2d-2 \Big )  = \Big \lceil \frac{n}{2} \Big \rceil +t +2d-1,\]
i.e., the consecutive to the first element of $-B$. Hence, $A+B$ precisely consists of the intervals of length $d-1$ that fill the gaps between the elements of $-B$ and one more interval that follows the last element of $-B$. Hence, the sets $-B$ and $A+B$ are disjoint, and their union is an interval starting at the first element of $-B$. The last element of this interval is the last element of $A+B$ which equals
\[ \Big( \frac{1}{2} \Big ( \Big \lceil \frac{n}{2} \Big \rceil +t+1 \Big )+d-2 \Big ) + \Big ( \Big \lfloor \frac{n}{2} \Big \rfloor - t -2d+2 \Big)  = \frac{1}{2} \Big ( \Big \lfloor \frac{3n}{2} \Big \rfloor -t-1 \Big ) -d +1,\]
where we have used~\eqref{eq:b_l} and the fact that $\frac{1}{2} \cdot \lceil \frac{n}{2} \rceil + \lfloor \frac{n}{2} \rfloor = \frac{1}{2} \cdot \lfloor \frac{3n}{2} \rfloor$.
\end{proof}

In the following claim we consider the set $B+C$.

\begin{claim}\label{claim:B+C}
If $|C| \geq d$ then the set $B+C$ is
\[ \Big [ \frac{1}{2} \Big ( \Big \lfloor \frac{3n}{2} \Big \rfloor-t+1 \Big ) + 2d-2 ,   n-2d+2 \Big ] .\]
\end{claim}

\begin{proof}
As $|B| = k-3$ and $C$ is an interval, $B+C$ is a union of $k-3$ intervals, each of which is a shift of $C$ by an element of $B$.
Since $B$ is an arithmetic progression of difference $d$, each of these intervals is a shift by $d$ of the previous interval, and since $|C| \geq d$, these shifts overlap, so $B+C$ is itself an interval. Its first element is
\[ \Big ( \frac{1}{2} \Big ( \Big \lceil \frac{n}{2} \Big \rceil+t+1 \Big ) + 2d-2 \Big )+ \Big( \Big \lfloor \frac{n}{2} \Big \rfloor -t \Big) =  \frac{1}{2} \Big ( \Big \lfloor \frac{3n}{2} \Big \rfloor-t+1 \Big ) + 2d-2,\]
and the last one, using~\eqref{eq:b_l}, is
\[ \Big ( \Big \lfloor \frac{n}{2} \Big \rfloor - t -2d+2 \Big) + \Big( \Big \lceil \frac{n}{2} \Big \rceil +t \Big ) =  n-2d+2,\]
so we are done.
\end{proof}

We are now ready to prove Theorem~\ref{thm:complete_small}.
\paragraph{Proof of Theorem~\ref{thm:complete_small}.}
Denote $S = S^{(n)}_{t,d,k} = (\pm A) \cup (\pm B) \cup C$. Assume that $|C| \geq d$.
The proof is given in the following two lemmas.
\begin{lemma}
The set $S \subseteq \Z_n$ is complete.
\end{lemma}

\begin{proof}
We shall prove that $\Z_n \subseteq S \cup (S+S)$. By Claim~\ref{claim:sumfree_short}, using the symmetry of $S$ and the fact that $0 \notin S$, it suffices to prove that every element of $[ \lfloor \frac{n}{2} \rfloor ,n-1]$ is in either $S$ or $S+S$. This is shown via the following steps:
\begin{itemize}
  \item $[\lfloor \frac{n}{2} \rfloor, \lceil \frac{n}{2} \rceil + t]$: This set is contained in $C \subseteq S$.
  \item $[\lceil \frac{n}{2} \rceil +t+1, \lceil \frac{n}{2} \rceil + t+2d-3]$: Observe that this set is equal to $A+A \subseteq S + S$.
  \item $[ \lceil \frac{n}{2} \rceil + t+2d-2, \frac{1}{2}( \lfloor \frac{3n}{2} \rfloor -t-1)-d+1]$: By Claim~\ref{claim:-B,A+B}, this set is the union of $-B$ and $A+B$, and is thus contained in $S \cup (S+S)$.
  \item $[ \frac{1}{2}( \lfloor \frac{3n}{2} \rfloor -t-1)-d+2, \frac{1}{2}(\lfloor \frac{3n}{2} \rfloor-t-1)]$: This set is equal to $-A \subseteq S$.
  \item $[ \frac{1}{2}(\lfloor \frac{3n}{2} \rfloor-t+1), \frac{1}{2}(\lfloor \frac{3n}{2} \rfloor-t+1) + d+|C|-3]$: Observe that this set is equal to $A+C \subseteq S +S$.
  \item $[ \frac{1}{2}(\lfloor \frac{3n}{2} \rfloor-t+1) + d+|C|-2, n-2d+2 ]$: By Claim~\ref{claim:B+C}, using $|C| \geq d$, this interval is contained in $B+C \subseteq S+S$.
  \item $[n-2d+3,n-d]:$ For $d=2$ this set is empty. For $d \geq 3$ it is contained in the interval $[n-2d+2,n-d]$, obtained by adding to $A$ the negative of the first element of $B$, and is thus contained in $S+S$.
  \item $[n-d+1,n-1]:$ Observe that by $|C| \geq d$ we have
  \[ 2 \Big \lfloor \frac{n}{2} \Big \rfloor -2t \leq n-d+1,\]
  hence this set is contained in $[2 \lfloor \frac{n}{2} \rfloor -2t,n-1] \subseteq C+C \subseteq S+S$.
\end{itemize}
This completes the proof of the lemma.
\end{proof}

\begin{lemma}
The set $S \subseteq \Z_n$ is sum-free.
\end{lemma}

\begin{proof}
By Claim~\ref{claim:sumfree_short}, using the symmetry of $S$ and $C$, it suffices to verify that each of the sets
\[ A+A,~A+B,~B+B,~B+C,~A+C, \mbox{~and~}~C+C \]
has an empty intersection with $S$. We verify this below:
\begin{itemize}
  \item The set $A+A$ is equal to $[\lceil \frac{n}{2} \rceil +t+1, \lceil \frac{n}{2} \rceil + t+2d-3]$, hence its elements are located immediately after $C$ and before the first element $-b_l$ of $-B$, so they do not belong to $S$.
  \item By Claim~\ref{claim:-B,A+B}, the set $A+B$ is equal to $[ \lceil \frac{n}{2} \rceil + t+2d-2, \frac{1}{2}( \lfloor \frac{3n}{2} \rfloor -t-1)-d+1] \setminus (-B)$, hence its elements are located after $C$, before $-A$, and do not include elements of $-B$, so they do not belong to $S$.
  \item The set $B+B$ is a sum of an arithmetic progression of difference $d$ and length $k-3$ with itself, hence it forms an arithmetic progression of difference $d$ of length $2k-7$. The first element of $B+B$ is $\lceil \frac{n}{2} \rceil+t+4d-3$, located on the right of $C$, and its last element, using~\eqref{eq:b_l}, is $2 \lfloor \frac{n}{2} \rfloor -2t-4d+4 < n$. Hence, it suffices to verify that $B+B$ does not intersect $-B$ and $-A$. The first element of $-B$ is $\lceil \frac{n}{2} \rceil+t+2d-2$. Observe that it differs from the first element of $B+B$ by $2d-1$, which is not divisible by $d$. Since $-B$ is an arithmetic progression of difference $d$, it follows that $B+B$ does not intersect it.
      Now, let $a_0$ denote the first element of $A$.
      Using~\eqref{eq:b_l}, the sum of the first and last elements of $B$ is
      \[  \Big( \frac{1}{2} \Big (\Big \lceil \frac{n}{2} \Big \rceil+t+1 \Big ) + 2d-2  \Big) + \Big( \Big \lfloor \frac{n}{2} \Big \rfloor - t -2d+2 \Big) = \frac{1}{2} \Big( \Big \lfloor \frac{3n}{2} \Big \rfloor -t+1 \Big ) = -a_0+1. \]
      This implies that the element consecutive to the interval $-A = [-a_0-d+2,-a_0]$ lies in $B+B$.
      Since the size of $-A$ is $d-1$ and the arithmetic progression $B+B$ has difference $d$, we get that $B+B$ does not intersect $-A$.
  \item By Claim~\ref{claim:B+C}, the set $B+C$ is equal to $[ \frac{1}{2} (\lfloor \frac{3n}{2} \rfloor-t+1 ) + 2d-2 ,   n-2d+2 ]$, hence its elements are located on the right of $-A$ and do not belong to $S$.
  \item The set $A+C$ is equal to $[ \frac{1}{2}(\lfloor \frac{3n}{2} \rfloor-t+1), \frac{1}{2}(3\lceil \frac{n}{2} \rceil+3t+1) + d-2]$. It can be verified that the values of $n$ in~\eqref{eq:n_values} satisfy that the right edge of the interval is smaller than $n$. Hence, its elements are located on the right of $-A$ and do not belong to $S$.
  \item The set $C+C$ is equal to $[0, |C|-1] \cup [n-|C|+1,n-1]$. By~\eqref{eq:n_values}, it is clear that $n \geq 6t+18$, hence the first element of $A$ satisfies
  \[a_0 = \frac{1}{2} \Big ( \Big \lceil \frac{n}{2} \Big  \rceil+t+1 \Big) \geq \frac{1}{2}(4t+10) = 2t+5 > |C|. \]
  It follows that the elements of $C+C$ are located on the left of $A$ and on the right of $-A$, so they do not belong to $S$.
\end{itemize}
This completes the proof of the lemma.
\end{proof}

We can now show that for every sufficiently large integer $n$ there exist symmetric complete sum-free subsets of $\Z_n$ of various sizes.

\begin{theorem}\label{thm:various}
There exist constants $c_1,c_2,c_3>0$ such that for every sufficiently large integer $n$ there exists a collection of symmetric complete sum-free subsets of $\Z_n$ whose sizes form an arithmetic progression with first element at most $c_1 \cdot \sqrt{n}$, difference at most $c_2 \cdot \sqrt{n}$, and last element at least $\frac{n}{3}-c_3 \cdot \sqrt{n}$.
\end{theorem}

\begin{proof}
Let $n$ be a sufficiently large integer.
Define $d_0$ to be the unique integer in $\lfloor \sqrt{n} \rfloor - \{0,1,2\}$ that satisfies $d_0 = 1~(\mod ~3)$, and notice that $\sqrt{n} - 3 \leq d_0 \leq \sqrt{n}$.
Define $a=11$ if $n$ is odd and $a=14$ if $n$ is even. As $n$ is sufficiently large, one can write $n = 2m-a$ for a positive integer $m$.
Let $m'$ be the integer satisfying $m' = m ~(\mod~2d_0)$ and $0 \leq m' < 2d_0$.
The choice of $d_0$ guarantees that $3$ divides one of $m'+2d_0, m'+4d_0, m'+6d_0$. Denote by $3t_0$ the one divided by $3$, and notice that $2d_0 \leq 3t_0 \leq 8d_0$, implying $2 \sqrt{n}/3 -2 \leq t_0 \leq 8\sqrt{n}/3$ and $d_0 \leq 2t_0+1$.
As $m = 3t_0~(\mod~2d_0)$, it follows that there exists an integer $k_0$ for which $m = 2d_0k_0 +3t_0$, thus
\begin{eqnarray}\label{eq:n}
n = 4d_0k_0 +6t_0-a.
\end{eqnarray}
Using again the assumption that $n$ is sufficiently large, we have $t_0 \geq 1$, $d_0 \geq 2$, $k_0 \geq 4$ and $k_0 \leq n/(4d_0) \leq \sqrt{n}/3$.
Applying Theorem~\ref{thm:complete_small} to the integers $t_0, d_0, k_0$, we get that there exists a symmetric complete sum-free subset of $\Z_n$, whose size $s$, using~\eqref{eq:size_S}, satisfies $s = 2(d_0+k_0+t_0)-7$ if $n$ is even and $s = 2(d_0+k_0+t_0)-6$ if $n$ is odd. In both cases we have $s \leq c_1 \cdot \sqrt{n}$ for a universal constant $c_1>0$.

Now, define $b = \Big \lfloor \frac{k_0-4}{3} \Big \rfloor$.
For every integer $0 \leq i \leq b$ denote
\[k_i = k_0-3 \cdot i \mbox{~~~and~~~} t_i = t_0+2d_0 \cdot i. \]
Observe, using~\eqref{eq:n}, that for every such $i$ we have $n = 4d_0k_i +6t_i-a$, $k_i \geq k_0-3b \geq 4$, and $d_0 \leq 2t_i+1$.
Applying Theorem~\ref{thm:complete_small} to the integers $t_i, d_0, k_i$, we get that there exists a symmetric complete sum-free subset of $\Z_n$, whose size, as follows from~\eqref{eq:size_S}, is $s+i \cdot 2(2d_0-3)$. Denote $r = 2(2d_0-3)$, and notice that $r \leq c_2 \cdot \sqrt{n}$ for a universal constant $c_2>0$. The integers $s+i \cdot r$, for $0 \leq i \leq b$, form an arithmetic progression, each of its elements is a size of a symmetric complete sum-free subset of $\Z_n$. To complete the proof, observe that its last element satisfies
\[ s+ b \cdot r \geq \frac{k_0-6}{3} \cdot 2(2d_0-3) = \frac{1}{3}(4d_0k_0-24d_0-6k_0+36) \geq \frac{n}{3} - c_3 \cdot \sqrt{n}\]
for a universal constant $c_3>0$, and we are done.
\end{proof}

Theorem~\ref{thm:SmallIntro} is an immediate consequence of Theorem~\ref{thm:various}. We turn to derive Theorem~\ref{thm:DenseIntro}.

\begin{proof}[ of Theorem~\ref{thm:DenseIntro}]
Let $0 \leq \alpha \leq \frac{1}{3}$ and $\eps >0$ be fixed constants.
Let $n$ be a sufficiently large integer.
By Theorem~\ref{thm:various}, for universal constants $c_1,c_2,c_3>0$,
there exists a collection of symmetric complete sum-free subsets of $\Z_n$ whose sizes form an arithmetic progression with first element at most $c_1 \cdot \sqrt{n}$, difference at most $c_2 \cdot \sqrt{n}$, and last element at least $\frac{n}{3}-c_3 \cdot \sqrt{n}$.
Let $S$ be a set in this collection whose size is closest to $\alpha \cdot n$. Then, for $c = \max(c_1,c_2/2,c_3)$, we have
\[\alpha \cdot n - c \cdot \sqrt{n} \leq |S| \leq \alpha \cdot n + c \cdot \sqrt{n}.\]
Assuming that $n$ is sufficiently large, so that $\eps \geq c/\sqrt{n}$, it follows that
\[ \alpha -\eps \leq \frac{|S|}{n} \leq \alpha+\eps,\]
as required.
\end{proof}

\section*{Acknowledgements}
We thank Peter Cameron, Grahame Erskine, David Grynkiewicz, Sarah Hart, Karen Seyffarth and Olof Sisask for their kind help with the literature and for useful discussions.

\bibliographystyle{abbrv}

\end{document}